\documentclass[runningheads,orivec]{llncs}
\usepackage[T1]{fontenc}
\usepackage{graphicx}
\usepackage{geometry}
\usepackage{mathtools}
\usepackage{amssymb}
\usepackage{tikz}
\usepackage{subcaption}
\usepackage{setspace}
\usepackage[color=white]{todonotes}
\usepackage{appendix}
\usepackage[colorlinks,citecolor=green!60!black]{hyperref}

\usepackage{amsthm}

\usepackage{xcolor}

\usetikzlibrary{decorations.pathmorphing,decorations.pathreplacing,decorations.markings}
\usetikzlibrary{arrows,arrows.meta}

\newtheoremstyle{thmpart}
{3pt}
{3pt}
{}
{}
{\normalfont\itshape}
{:}
{.5em}
{\thmname{#1}\thmnumber{ #2}\thmnote{ (#3)}}

\theoremstyle{thmpart}

\newcommand{\overbar}[1]{\mkern 1.5mu\overline{\mkern-1.5mu#1\mkern-1.5mu}\mkern 1.5mu}

\setuptodonotes{inline}

\makeatletter
\def\@seccntformat#1{\@ifundefined{#1@cntformat}%
   {\csname the#1\endcsname\quad}  
   {\csname #1@cntformat\endcsname}
}
\let\oldappendix\appendix 
\renewcommand\appendix{%
    \oldappendix
    \newcommand{\section@cntformat}{\appendixname~\thesection\quad}
}
\makeatother

\begin{document}
\title{Some short notes on oriented line graphs and related matrices}
%
%
\author{
Jacob Antony\inst{1} \and 
Cyriac Antony\inst{2} \and 
Jinitha Varughese\inst{3} \and 
Bloomy Joseph\inst{4}
}
%
\authorrunning{Antony et al.}
%
\institute{
M.G. University, Kottayam, India; \email{phdmath2401@cmscollege.ac.in} \and 
Luiss University, Rome, Italy; \email{acyriac@luiss.it} \and
Bishop Kurialacherry College for Women, Kottayam, India; \email{jinith@gmail.com} \and 
Government Arts and Science College, Chelakkara, Thrissur, India; \email{bloomykoshy@gmail.com}
}
\maketitle              
\begin{abstract}
Oriented line graph, introduced by Kotani and Sunada (2000), is closely related to Hashimato's non-backtracking matrix (1989). It is known that for regular graphs $G$, the eigenvalues of the adjacency matrix of the oriented line graph $\vec{L}(G)$ of $G$ are the reciprocals of the poles of the Ihara zeta function of $G$. We determine the characteristic polynomial of the \( z \)-Hermitian adjacency matrix of \( \vec{L}(G) \) for each \( z\in \mathbb{C} \) and \( d \)-regular graph \( G \) with \( d\geq 3 \). 
Special cases of this matrix include the Hermitian adjacency matrix of \( \vec{L}(G) \) and the adjacency matrix of the underlying undirected graph of \( \vec{L}(G) \). 
We also exhibit an application to star coloring of graphs. 

\end{abstract}
\section{Introduction}
%
The notion of oriented line graphs is introduced by Kotani and Sunada~\cite{kotani_sunada}, and they are closely related to Hashimato's~\cite{hashimoto} non-backtracking matrix. 
Let \( G \) be a (simple undirected) graph. 
The \emph{oriented line graph} of \( G \) is the oriented graph with vertex set \( \bigcup_{uv\in E(G)}\{(u,v), (v,u)\} \) and there is an arc from a vertex \( (u,v) \) to a vertex \( (v,w) \) in it when \( u\neq w \)~\cite{parzanchevski}. 
We denote the oriented line graph of \( G \) by \( \vec{L}(G) \), and its underlying undirected graph by \( L^*(G) \). 
We can rephrase the definition of oriented line graph in terms of \( \vec{D}(G) \), the simple directed graph obtained from \( G \) be replacing each edge \( uv \) of \( G \) by two arcs \( (u,v) \) and \( (v,u) \) (i.e., replacing each edge by a digon). 
The oriented line graph of \( G \) is the oriented graph with the set of arcs of \( \vec{D}(G) \) as its vertex set, and there is an arc in it from a vertex \( (u,v) \) to a vertex \( (v,w) \) when \( u\neq w \) (see Fig.~\ref{fig:eg Larrow}). 
The non-backtracking matrix of \( G \) is the adjacency matrix of~\( \vec{L}(G) \).

\begin{figure}[hbt]
\centering
\tikzset{
>={Straight Barb[length=1.0mm]},
circnode/.style={draw,circle,minimum size=5pt},
}
\begin{subfigure}[c]{0.5\textwidth}
\centering
\scalebox{0.7}{%
\begin{tikzpicture}[thick]
\path (0:2) node(1)[circnode]{};
\path (120:2) node(2)[circnode]{};
\path (-120:2) node(3)[circnode]{};
\path (1) --+(3.5,0) node(0)[circnode]{};

\draw (0)--(1)--(2)--(3)--(1);
\end{tikzpicture}
}
\caption{\( G \)}
\end{subfigure}%
\begin{subfigure}[c]{0.5\textwidth}
\centering
\scalebox{0.7}{%
\begin{tikzpicture}[thick]
\path (0:2) node(1)[circnode]{};
\path (120:2) node(2)[circnode]{};
\path (-120:2) node(3)[circnode]{};
\path (1) --+(3.5,0) node(0)[circnode]{};

\begin{scope}[thick,decoration={
    markings,
    mark=at position 0.525 with {\arrow{>}}}
    ] 
\draw [postaction={decorate}] (0) to[bend left=20] node(01)[pos=0.5]{} (1);
\draw [postaction={decorate}] (1) to[bend left=20] node(10)[pos=0.5]{} (0);
\draw [postaction={decorate}] (1) to[bend left=20] node(12)[pos=0.5]{} (2);
\draw [postaction={decorate}] (2) to[bend left=20] node(21)[pos=0.5]{} (1);
\draw [postaction={decorate}] (1) to[bend left=20] node(13)[pos=0.5]{} (3);
\draw [postaction={decorate}] (3) to[bend left=20] node(31)[pos=0.5]{} (1);
\draw [postaction={decorate}] (2) to[bend left=20] node(23)[pos=0.5]{} (3);
\draw [postaction={decorate}] (3) to[bend left=20] node(32)[pos=0.5]{} (2);
\end{scope}
\end{tikzpicture}
}
\caption{\( \vec{D}(G) \)}
\end{subfigure}%

\begin{subfigure}[c]{0.5\textwidth}
\centering
\scalebox{0.7}{%
\begin{tikzpicture}[thick]
\begin{scope}[opacity=0]
\path (0:2) node(1)[circnode]{};
\path (120:2) node(2)[circnode]{};
\path (-120:2) node(3)[circnode]{};
\path (1) --+(3.5,0) node(0)[circnode]{};
\end{scope}

\begin{scope}[thick,opacity=0,decoration={
    markings,
    mark=at position 0.525 with {\arrow{>}}}
    ] 
\draw [postaction={decorate}] (0) to[bend left=20] node(01)[pos=0.5]{} (1);
\draw [postaction={decorate}] (1) to[bend left=20] node(10)[pos=0.5]{} (0);
\draw [postaction={decorate}] (1) to[bend left=20] node(12)[pos=0.5]{} (2);
\draw [postaction={decorate}] (2) to[bend left=20] node(21)[pos=0.5]{} (1);
\draw [postaction={decorate}] (1) to[bend left=20] node(13)[pos=0.5]{} (3);
\draw [postaction={decorate}] (3) to[bend left=20] node(31)[pos=0.5]{} (1);
\draw [postaction={decorate}] (2) to[bend left=20] node(23)[pos=0.5]{} (3);
\draw [postaction={decorate}] (3) to[bend left=20] node(32)[pos=0.5]{} (2);
\end{scope}

\node (01-) at (01) [circnode]{};
\node (10-) at (10) [circnode]{};
\node (12-) at (12) [circnode]{};
\node (21-) at (21) [circnode]{};
\node (13-) at (13) [circnode]{};
\node (31-) at (31) [circnode]{};
\node (23-) at (23) [circnode]{};
\node (32-) at (32) [circnode]{};

\begin{scope}[thick,decoration={
    markings,
    mark=at position 0.5 with {\arrow{>}}}
    ] 
\draw [postaction={decorate}] (12-)--(23-);
\draw [postaction={decorate}] (23-)--(31-);
\draw [postaction={decorate}] (31-)--(12-);
\draw [postaction={decorate}] (21-)--(13-);
\draw [postaction={decorate}] (13-)--(32-);
\draw [postaction={decorate}] (32-)--(21-);
\draw [postaction={decorate}] (01-) to[bend right](12-);
\draw [postaction={decorate}] (31-) to[bend right] (10-);
\draw [postaction={decorate}] (21-) to[bend left] (10-);
\draw [postaction={decorate}] (01-) to[bend left](13-);
\end{scope}
\end{tikzpicture}
}
\caption{\( \vec{L}(G) \)}
\end{subfigure}%
\begin{subfigure}[c]{0.5\textwidth}
\centering
\scalebox{0.7}{%
\begin{tikzpicture}[thick]
\begin{scope}[draw=gray,fill=gray]
\path (0:2) node(1)[circnode]{};
\path (120:2) node(2)[circnode]{};
\path (-120:2) node(3)[circnode]{};
\path (1) --+(3.5,0) node(0)[circnode]{};
\end{scope}

\begin{scope}[thick,draw=gray,fill=gray,decoration={
    markings,
    mark=at position 0.525 with {\arrow{>}}}
    ] 
\draw [postaction={decorate}] (0) to[bend left=20] node(01)[pos=0.5]{} (1);
\draw [postaction={decorate}] (1) to[bend left=20] node(10)[pos=0.5]{} (0);
\draw [postaction={decorate}] (1) to[bend left=20] node(12)[pos=0.5]{} (2);
\draw [postaction={decorate}] (2) to[bend left=20] node(21)[pos=0.5]{} (1);
\draw [postaction={decorate}] (1) to[bend left=20] node(13)[pos=0.5]{} (3);
\draw [postaction={decorate}] (3) to[bend left=20] node(31)[pos=0.5]{} (1);
\draw [postaction={decorate}] (2) to[bend left=20] node(23)[pos=0.5]{} (3);
\draw [postaction={decorate}] (3) to[bend left=20] node(32)[pos=0.5]{} (2);
\end{scope}

\node(01-) at (01) [circnode]{};
\node(10-) at (10) [circnode]{};
\node(12-) at (12) [circnode]{};
\node(21-) at (21) [circnode]{};
\node(13-) at (13) [circnode]{};
\node(31-) at (31) [circnode]{};
\node(23-) at (23) [circnode]{};
\node(32-) at (32) [circnode]{};

\begin{scope}[thick,decoration={
    markings,
    mark=at position 0.5 with {\arrow{>}}}
    ] 
\draw [postaction={decorate}] (12-)--(23-);
\draw [postaction={decorate}] (23-)--(31-);
\draw [postaction={decorate}] (31-)--(12-);
\draw [postaction={decorate}] (21-)--(13-);
\draw [postaction={decorate}] (13-)--(32-);
\draw [postaction={decorate}] (32-)--(21-);
\draw [postaction={decorate}] (01-) to[bend right](12-);
\draw [postaction={decorate}] (31-) to[bend right] (10-);
\draw [postaction={decorate}] (21-) to[bend left] (10-);
\draw [postaction={decorate}] (01-) to[bend left](13-);
\end{scope}
\end{tikzpicture}
}
\caption{\( \vec{L}(G) \) drawn on top of \( \vec{D}(G) \)}
\end{subfigure}%
\caption{An example of the oriented line graph operation.}
\label{fig:eg Larrow}
\end{figure}

For \( z\in \mathbb{C} \) and an oriented graph \( \vec{H} \) with vertex set \( \{v_1,v_2,\dots,v_n\} \), the \emph{\( z \)-Hermitian adjacency matrix} of \( \vec{H} \) is the matrix whose \( (i,j) \)-th entry is \( 1 \) if \( (v_i,v_j),(v_j,v_i)\in E(\vec{H}) \); \( z \) if \( (v_i,v_j)\in E(\vec{H}) \) but \( (v_j,v_i)\notin E(\vec{H}) \); \( \bar{z} \) if \( (v_j,v_i)\in E(\vec{H}) \) but \( (v_i,v_j)\notin E(\vec{H}) \); and 0 if \( (v_i,v_j),(v_j,v_i)\notin E(\vec{H}) \).

The characteristic polynomial of the adjacency matrix of \( \vec{L}(G) \) is well-known for \( d \)-regular graphs \( G \) with \( d\geq 3 \). 
It is known that the roots of this polynomial are the reciprocals of the poles of the Ihara zeta function of \( G \)~\cite{hashimoto}.
We determine the characteristic polynomial of the \( z \)-Hermitian adjacency matrix of \( L^*(G) \) for each \( z\in \mathbb{C} \) and \( d \)-regular graph \( G \) with \( d\geq 3 \). 
Special cases of this matrix include the Hermitian adjacency matrix of \( \vec{L}(G) \) and the adjacency matrix of \( L^*(G) \). 
It is interesting that for a \( d \)-regular graph \( G \) with \( d\geq 3 \), the line graph has the characteristic polynomial \( \operatorname{char}(L(G);x)=(x+2)^{m-n}\prod_{i=1}^{n} (x-\lambda_i-d+2) \), whereas \( L^*(G) \) has the characteristic polynomial \( \operatorname{char}(L^*(G);x)=
(x+2)^{m-n}(x-2)^{m-n}\prod_{i=1}^{n} (x-\lambda_i-d+2)(x-\lambda_i+d-2) \).
We also exhibit an application to star coloring of graphs (see Section~\ref{sec:appln star col} for definition).


%
All graphs we consider are simple and finite; they are also undirected unless otherwise specified. 
To clarify, the directed graphs we consider are also simple (i.e., no parallel arcs, although two arcs in opposite directions could be present). 
We denote the first projection map (i.e., \( (x,y)\mapsto x \)) by \( \operatorname{proj_1} \), and the second projection map by \( \operatorname{proj_2} \). 
For convenience, we assume that the vertex set of the complete graph \( K_q \) is \( \mathbb{Z}_q \) for all \( q\in \mathbb{Z} \). 
Throughout this paper, \( G \) denotes a graph on which an operation, such as line graph or oriented line graph operation, is applied, and we denote the number of vertices and number of edges of \( G \) by \( n \) and \( m \), respectively. 


\section[\( z \)-Hermitian Spectrum of \( \vec{L}(G) \)]{\boldmath \( z \)-Hermitian Spectrum of \( \vec{L}(G) \)}\label{sec:alpha-Hermitian spectrum}
We determine the characteristic polynomial of the \( z \)-Hermitian adjacency matrix of \( \vec{L}(G) \) for each \( z\in \mathbb{C} \) and \( d \)-regular graph \( G \) with \( d\geq 3 \). 

Lubetzky and Peres~\cite{lubetzky_peres} proved that the adjacency matrix of an oriented line graph is unitarily similar to a block-diagonal matrix with \( 1\times 1 \) and \( 2\times 2 \) blocks. 
\begin{theorem}[Lubetzky and Peres~\cite{lubetzky_peres}]\label{thm:adj Larrow}
Let \( G \) be a connected \( d \)-regular graph with \( n \) vertices,  \( m \) edges and eigenvalues \( \lambda_1,\lambda_2,\dots,\lambda_n \), where \( d\geq 3 \). 
Then, the non-backtracking matrix of \( G \) is unitarily similar to the following block-diagonal matrix~\( C \), 
\begin{equation}
  \label{eq:B-almost-diagonalization}
   C=\operatorname{diag}\left( 
\begin{bmatrix}
\theta_1 & \alpha_1  \\
0 & \theta'_1
\end{bmatrix},\ldots, \begin{bmatrix}
\theta_n & \alpha_n
 \\ 0 & \theta'_n
\end{bmatrix},-1,\dots,-1,1,\dots,1 \right) 
\end{equation}
where \( -1 \) and \( 1 \) are repeated \( m-n \) times, 
\[ |\alpha_i|=
\begin{cases}
d-2 \text{ if } |\lambda_i|\leq 2\sqrt{d-1}\\
\sqrt{d^2-\lambda_i^2} \text{ otherwise}
\end{cases}
\]
and $\{\theta_i,\theta_i'\}=\{(\lambda/2)+\sqrt{(\lambda/2)^2-(d-1)},(\lambda/2)-\sqrt{(\lambda/2)^2-(d-1)}\}$ for $1\leq i\leq n$.
\end{theorem}

Using Theorem~\ref{thm:adj Larrow}, we determine the \( z \)-Hermitian spectrum of \( \vec{L}(G) \). 
\begin{theorem}\label{thm:Hermitian spectrum}
Let \( G \) be a connected \( d \)-regular graph with \( n \) vertices,  \( m \) edges and eigenvalues\\ \( \lambda_1,\lambda_2,\dots,\lambda_n \), where \( d \geq 3 \). 
Then, 
the $z$-Hermitian characteristic polynomial of \( \vec{L}(G) \) where \( z \in \mathbb{C} \) is 
\[
(x^2-(z+\bar{z})^2)^{m-n}\prod_{i=1}^{n} \left[(x-z\lambda_i)(x-\bar{z}\lambda_i)-|z|^2 d^2 +(z+\bar{z})^2(d-1)\right] 
\]
\end{theorem}
\begin{proof}[Proof overview]
    Let \(B\) be the non-backtracking matrix of \(G\). Then \(B\) is unitarily similar to the block diagonal matrix \(C\) in Theorem~\ref{thm:adj Larrow}. 
    Thus \( \theta_j+\theta_j' = \lambda_j\), \quad
    \( \theta_j \theta_j' = (d-1)\),
    and \( \theta_j^2 + \theta_j'^2 = \lambda_j^2-2(d-1)\).
    
    Since \(zB+\bar{z}B^\ast = z(P^\ast CP) + \bar{z}(P^\ast CP)^\ast = zP^\ast CP + \bar{z} P^\ast C^\ast P = P^\ast zC P + P^\ast \bar{z}C^\ast P = P^\ast (zC+\bar{z}C^\ast) P\), the \(z\)-Hermitian adjacency matrix of \(L^\ast(G)\) is unitarily similar to \(zC+\bar{z}C^\ast=\)
    \(diag\left(
    \begin{bmatrix}
        z\theta_1+\bar{z}\bar{\theta}_1 & z\alpha_1 \\ \bar{z}\bar{\alpha}_1 & z\theta_1'+\bar{z}\bar{\theta}_1'
    \end{bmatrix},
    \dotsc,
    \right.\)\\
    \(\left.
    \begin{bmatrix}
            z\theta_n+\bar{z}\bar{\theta}_n & z\alpha_n \\ \bar{z}\bar{\alpha}_n & z\theta_n'+\bar{z}\bar{\theta}_n'
    \end{bmatrix},-z-\bar{z},\dotsc,-z-\bar{z},z+\bar{z},\dotsc,z+\bar{z}
    \right)\). 
    
    For \( \lambda_j \le 2\sqrt{d-1} \), \(|\alpha_j|^2=(d-2)^2\) and \( \theta_j'=\bar{\theta}_j\).
    Therefore,
    \[
(x^2-(z+\bar{z})^2)^{m-n}\prod_{i=1}^{n} \left( (x-z\lambda_i)(x-\bar{z}\lambda_i)-|z|^2 d^2 +(z+\bar{z})^2(d-1)\right) 
\]

For \( \lambda_j > 2\sqrt{d-1}\), \(|\alpha_j|^2 = d^2-\lambda_j^2\) and \( \bar{\theta}_j = \theta_j\).
Therefore,
    \[
(x^2-(z+\bar{z})^2)^{m-n}\prod_{i=1}^{n} \left( (x-z\lambda_i)(x-\bar{z}\lambda_i)-|z|^2 d^2 +(z+\bar{z})^2(d-1)\right) 
\]
\end{proof}

\begin{corollary}
The \( \omega \)-Hermitian characteristic polynomial of \( \vec{L}(G) \)  is
\[ (x^2-1)^{m-n}\prod_{i=1}^{n} \left[ (x^2+x\lambda_i+\lambda_i^2)-(d^2-d+1) \right] \qquad 
\] 
\end{corollary}
\begin{corollary}\label{cor:char poly Lstar}
The characteristic polynomial of \( L^\ast(G) \) is
\[ (x^2-4)^{m-n}\prod_{i=1}^{n} \left( (x-\lambda_i)^2-(d-2)^2 \right)  \]
\end{corollary}
A direct proof of Corollary~\ref{cor:char poly Lstar} is provided in Appendix. 
\begin{corollary}
If a connected \( d \)-regular graph \( G \) is integral and \( d\geq 3 \), then \( L^*(G) \) is integral.
\end{corollary}

\noindent
Remark: 
\( \operatorname{char}(L(G);x)=(x+2)^{m-n}\prod_{i=1}^{n} \left( (x-\lambda_i)-(d-2) \right) 
\), 
and\\
\( \operatorname{char}(L(\vec{D}(G));x)=x^{m-n}\prod_{i=1}^{n} (x-\lambda_i) \)
\cite{beineke_bagga2021a} 
(note that \( L(\vec{H}) \) denotes the line digraph of a digraph \( \vec{H} \)).

The skew-symmetric adjacency matrix of \( \vec{L}(G) \) can be determined similarly. 
The proof is omitted (hint: \( B-B^t=P^*(C-C^*)P \)). 
\begin{theorem}
Let \( G \) be a connected \( d \)-regular graph with \( n \) vertices,  \( m \) edges and eigenvalues\\ \( \lambda_1,\lambda_2,\dots,\lambda_n \), where \( d\geq 3 \). 
Then, the characteristic polynomial of the skew-symmetric adjacency matrix of \( \vec{L}(G) \) is \( x^{2(m-n)} \prod_{i=1}^{n} (x^2+d^2-\lambda_i^2) \).
\qed
\end{theorem}
\noindent

The next theorem follows from Lemma~3.1 of Bilu and Linial~\cite{bilu_linial} since \( L^*(G) \) is a 2-lift of \( L(G) \)~\cite{cyriac_shalu1}, and the characteristic polynomials of \( L^*(G) \) and \( L(G) \) are 
\( (x^2-4)^{m-n}\prod_{i=1}^{n} \left( (x-\lambda_i)^2-(d-2)^2 \right) \) and \( (x+2)^{m-n}\prod_{i=1}^{n} \left( (x-\lambda_i)-(d-2) \right) \),\\ respectively. 
\begin{theorem}\label{thm:signed assoc Lstar}
Let \( G \) be a connected \( d \)-regular graph with \( n \) vertices,  \( m \) edges and eigenvalues\\ \( \lambda_1,\lambda_2,\dots,\lambda_n \), where \( d\geq 3 \). 
Let \( \psi \) be an LBH from \( L^*(G) \) to \( L(G) \). 
Let \( \{V_0,V_1\} \) be a partition of the vertex set of \( \vec{L}(G) \) such that \( |\psi^{-1}(v)\cap V_i|=1 \) for each vertex \( v \) of \( G \) and each \( i\in \mathbb{Z}_2 \). 
Let \( (L(G),s) \) be the signed graph such that for each edge \( uv \) of \( L(G) \), we have \( s(uv)=1 \) if and only if \( uv \) belongs to the cut \( (V_0,V_1) \) in \( \vec{L}(G) \). 
Then, the characteristic polynomial of the signed adjacency matrix of \( (L(G),s) \) is \( (x-2)^{m-n}\prod_{i=1}^n (x-\lambda_i+d-2) \). 
\qed
\end{theorem}
\noindent
Remark: The existence and uniqueness of \( (L(G),s) \) is guaranteed (uniqueness up to switching).

\section{Consequence to star coloring}\label{sec:appln star col}
For a positive integer \( q \), a \emph{star \( q \)-coloring} of a graph \( G \) is a function \( f\colon V(G)\to \mathbb{Z}_q \) such that (i)~\( f(u)\neq f(v) \) for every edeg \( uv \) of \( G \), and (ii)~there is no path \( u,v,w,x \) in \( G \) with \( f(u)=f(w) \) and \( f(v)=f(x) \). 
A \emph{Locally Bijective Homomorphism (LBH)} from a graph \( G \) to a graph \( H \) is a mapping \( \psi\colon V(G)\to V(H) \) such that for every vertex \( v \) of \( G \), the restriction of \( \psi \) to the neighbourhood \( N_G(v) \) is a bijection from \( N_G(v) \) onto \( N_H(\psi(v)) \)~\cite{fiala_kratochvil}. 

For \( p\geq 2 \), a \( K_{1,p+1} \)-free \( 2p \)-regular graph \( G \) is star \( (p+2) \)-colorable if and only if \( G \) admits an LBH to \( L^*(K_{p+2}) \)~\cite{cyriac_shalu1}. 
Since existence of LBH from a graph \( H \) to a graph \( J \) implies that the characteristic polynomial of \( J \) divides the characteristic polynomial of \( H \)~\cite{fiala_kratochvil}, we have the following by Corollary~\ref{cor:char poly Lstar}. 
\begin{theorem}
Let \( G \) be a \( K_{1,p+1} \)-free \( 2p \)-regular star \( (p+2) \)-colorable graph, where \( p\geq 2 \). 
Then, the characteristic polynomial of \( ( \)the adjacency matrix of\( ) \) \( G \) is divisible by \( (x-2p) (x+2)^{(p-1)(p+2)/2} \) \( (x-2)^{p(p+1)/2} (x-p+2)^{p+1} (x+p)^{p+1}  \).
\qed
\end{theorem}

\subsubsection{\discintname}
The authors have no competing interests to declare that are relevant to the content of this article.


\bibliographystyle{splncs04}
\bibliography{myRefs14}

\appendix

\section{Direct proof for spectrum of \( L^*(G) \)}

\begin{theorem}\label{thm:Lstar char poly}
Let \( G \) be a connected \( d \)-regular graph with \( n \) vertices,  \( m \) edges and eigenvalues\\ \( \lambda_1,\lambda_2,\dots,\lambda_n \), where \( d\geq 3 \). 
Then, 
the characteristic polynomial of \( L^*(G) \) is exactly 
\( 
(x^2-4)^{m-n} \)\\ \( \prod_{i=1}^{n} \left( (x-\lambda_i)^2-(d-2)^2 \right) 
\).
\end{theorem}
\begin{proof}
Note that \( \theta_i+\theta_i'=\lambda_i \), and \( \theta_i \theta_i'=d-1 \) for all \( i \). 
Let \( B \) denote non-backtracking matrix of \( G \); that is, the adjacency matrix of \( \vec{L}(G) \). 
By Theorem~\ref{thm:adj Larrow}, \( B \) is unitarily similar to the matrix \( C \) of Equation~\ref{eq:B-almost-diagonalization}. 
That is, \( B=P^*CP \) for some unitary matrix \( P \) (i.e., \( P^{-1}=P^* \)). 
Note that \( B+B^t \) is the adjacency matrix of \( L^*(G) \). 
Also, 
\( B+B^t=B+B^*=P^*CP+(P^*CP)^*=P^*CP+P^*C^*P=P^*(C+C^*)P \). 
That is, \( B+B^t \) is unitarily similar to \( C+C^* \). 

\(
C^*=\operatorname{diag}\left(
\begin{bmatrix}
\overbar{\theta_1} & 0 \\
\overbar{\alpha_1} & \overbar{\theta_1'} 
\end{bmatrix},
\dots,
\begin{bmatrix}
\overbar{\theta_n} & 0 \\
\overbar{\alpha_n} & \overbar{\theta_n'} 
\end{bmatrix},
-1,\dots,-1,
1,\dots,1
\right)
\text{, and }
\)
~\\
\(
C+C^*=\operatorname{diag}\left(
\begin{bmatrix}
\theta_1+\overbar{\theta_1} & \alpha_1 \\
\overbar{\alpha_1} & \theta_1'+\overbar{\theta_1'} 
\end{bmatrix},
\dots,
\begin{bmatrix}
\theta_n+\overbar{\theta_n} & \alpha_n \\
\overbar{\alpha_n} & \theta_n'+\overbar{\theta_n'} 
\end{bmatrix},
-2,\dots,-2,
2,\dots,2
\right)
\),
~\\
\begingroup
where \( -1 \) and \( 1 \) (resp.\  \( -2 \) and \( 2 \)) are repeated \( m-n \) times. 
If \( |\lambda_i|\leq 2\sqrt{d-1} \), then \( \theta_i+\overbar{\theta_i}=\lambda_i=\theta_i'+\overbar{\theta_i'} \). 
If \( |\lambda_i|> 2\sqrt{d-1} \), then \( \theta_i+\overbar{\theta_i}=2\theta_i \), and \( \theta_i'+\overbar{\theta_i'}=2\theta_i' \). 
~\\
Since \( B+B^t \) is unitarily similar to \( C+C^* \), the required polynomial  
\(
\operatorname{char}(B+B^t;x)=\operatorname{char}(C+C^*;x)=det(xI-(C+C^*))=(x+2)^{m-n}(x-2)^{m-n}\prod_{i=1}^{n} h_i(x) \), \\
where \( h_i(x)=\left( (x-(\theta_i+\overbar{\theta_i}))(x-(\theta_i'+\overbar{\theta_i'}))-\alpha_i\overbar{\alpha_i}\right) 
\). 
Note that \( h_i(x)= 
x^2-x(\theta_i+\theta_i'+\overbar{\theta_i+\theta_i'})+(\theta_i+\overbar{\theta_i})(\theta_i'+\overbar{\theta_i'})-|\alpha_i|^2 = x^2-x(\lambda_i+\lambda_i)+(\theta_i+\overbar{\theta_i})(\theta_i'+\overbar{\theta_i'})-|\alpha_i|^2 \)
That is, \( h_i(x)=x^2-2\lambda_i x+(\theta_i+\overbar{\theta_i})(\theta_i'+\overbar{\theta_i'})-|\alpha_i|^2 \). 
Clearly, 
\( (\theta_i+\overbar{\theta_i})(\theta_i'+\overbar{\theta_i'})= \lambda_i^2 \) if \( |\lambda_i|\leq 2\sqrt{d-1} \), and \( (\theta_i+\overbar{\theta_i})(\theta_i'+\overbar{\theta_i'})= 4\theta_i\theta_i' \) otherwise. 
We know that \( |\alpha_i|^2=(d-2)^2 \) if \( |\lambda_i|\leq 2\sqrt{d-1} \) and \( |\alpha_i|^2=d^2-\lambda_i^2 \) otherwise.
Therefore, 
\endgroup
\[
h_i(x)=
\begin{cases}
x^2-2\lambda_ix+\lambda_i^2-(d-2)^2 \text{ if } |\lambda_i|\leq 2\sqrt{d-1}\\
x^2-2\lambda_ix+4(d-1)-(d^2-\lambda_i^2) \text{ otherwise}
\end{cases}
\]
Thus, in both cases, \( h_i(x)=x^2-2\lambda_ix+\lambda_i^2-(d-2)^2=(x-\lambda_i)^2-(d-2)^2 \). 
This completes the proof. 
\end{proof}

\end{document}